\documentclass[10pt]{article}
\usepackage{amsmath,amsfonts,graphicx,subcaption,amsthm,amssymb,authblk,appendix,hyperref}
\hypersetup{hidelinks}
\usepackage[ruled, vlined]{algorithm2e}
\usepackage[labelfont=bf]{caption}
\usepackage[total={6in, 8in}]{geometry}

\theoremstyle{definition}
\newtheorem{definition}{Definition}

\theoremstyle{definition}
\newtheorem{example}{Example}

\theoremstyle{plain}
\newtheorem{theorem}{Theorem}

\theoremstyle{plain}
\newtheorem{lemma}{Lemma}

\theoremstyle{plain}

\theoremstyle{plain}
\newtheorem{proposition}{Proposition}

\usepackage[style=ieee, citestyle=numeric, dashed=false]{biblatex}
\begin{document}

\title{Sparse Approximate Solutions to Max-Plus Equations with Application to Multivariate Convex Regression}

\author[1]{Nikos Tsilivis}
\author[2]{Anastasios Tsiamis}
\author[1]{Petros Maragos}
\affil[1]{School of ECE, National Technical University of Athens, Greece}
\affil[2]{ESE Department, SEAS, University of Pennsylvania, USA}
\maketitle

\begin{abstract}
    In this work, we study the problem of finding approximate, with minimum support set, solutions to matrix max-plus equations, which we call sparse approximate solutions. We show how one can obtain such solutions efficiently and in polynomial time for any $\ell_p$ approximation error. Based on these results, we propose a novel method for piecewise-linear fitting of convex multivariate functions, with optimality guarantees for the model parameters and an approximately minimum number of affine regions.
\end{abstract}

\section{Introduction}
In the last decades, the areas of signal and image processing had been greatly benefited from the advancement of the theory of sparse representations~\cite{Elad10}. Given a few linear measurements of an object of interest, sparse approximation theory provides efficient tools and algorithms for the acquisition of the sparsest (most zero elements) solution of the corresponding underdetermined linear system~\cite{Elad10,Nata95}. Based on the sparsity assumption of the initial signal, this allows perfect reconstruction from little data. Ideas stemming from this area had also given birth to \textit{compressed sensing} techniques~\cite{Dono06,CRT06} that allow accurate reconstructions from limited random projections of the initial signal, with wide-ranging applications in photography, magnetic resonance imaging and others.

Yet, there is a variety of problems in areas such as scheduling and synchronization~\cite{Cuni79,BCOQ01}, morphological image and signal analysis~\cite{Serr82,Heij94,Mara04} and optimization and optimal control~\cite{BCOQ01,AGG12,GMQ11} that do not admit linear representations. Instead, these problems share the ability to be described as a system of nonlinear equations, which involve maximum operations together with additions. The relevant theoretical framework has initially been developed in~\cite{Cuni79,BCOQ01,Butk10} and the appropriate algebra for this kind of problems is called \textit{max-plus} algebra. Motivated by the sparsity in the linear setting, \cite{TsMa19} introduced the notion of sparsity (signals with many $-\infty$ values, i.e. the identity element of this algebra) in max-plus algebra. Herein, we contribute to this theory, by studying the problem of sparse approximate solutions to matrix max-plus equations allowing the approximation error to be measured by any $l_p$ norm.

Subsequently, we apply our theoretical framework to the fundamental problem of multivariate convex regression, where the goal is to approximate a convex function by a piecewise-linear (PWL) one. Formulating the problem as a max-plus equation and computing a sparse solution enables us to obtain a PWL function with a \emph{minimum} number of affine regions.
In general,
the problem of fitting PWL functions has been studied before in many areas, including convex optimization, non-linear circuits, geometric programming, machine learning and statistics. Previous attempts on solving the multivariate version of it have focused on iterating between finding a suitable partition of the input space and locally fitting affine functions to each domain of the partition \cite{HKA16, MaBo09, KVY10, HaDu11}. A stable method is proposed in \cite{HaDu11}, where the authors propose a convex adaptive partitioning algorithm that is a consistent estimator and requires $\mathcal{O}(n(n+1)^2m\log(m)\log(\log(m)))$ computing time, where $n$ is the dimension of the input space and $m$ the number of points sampled from the convex function. Recently, it has been proposed to identify PWL functions with max-plus polynomials and formulate the regression problem as a max-plus equation, yielding a linear time algorithm \cite{MaTh20}. 

In summary, our contributions are the following: 
a) We pose a \emph{generalized} problem of finding the sparsest approximate solution to max-plus equations under a constraint which makes the problem more tractable, also known as the ``lateness constraint". The approximation error is in terms of any $\ell_p$ norm, for $p<\infty$. This formulation is more general than~\cite{TsMa19}, where only the $\ell_1$ norm was considered.
b) We prove that for any $\ell_p$, $p<\infty$, norm the problem has a supermodular structure, which allows us to solve it approximately but efficiently via a greedy algorithm, with a derived approximation ratio.
c) We investigate the $\ell_{\infty}$ case without the ``lateness constraint", reveal its hardness and propose a heuristic method for solving it.
d) We apply our framework to the problem of multivariate convex regression via PWL function fitting. Our method shares a common theoretical background with \cite{MaTh20}, but it differentiates from it as it allows an automatic, nearly optimal, selection of the affine regions, due to the imposed sparsity of the solutions. It, also, guarantees error bounds to the approximation, while compared to partitioning and locally fitting style methods \cite{HKA16, MaBo09, KVY10, HaDu11} it has lower complexity. \par

\section{Background Concepts}
For max and min operations we use the well-established lattice-theoretic symbols of $ \vee $ and $ \wedge $, respectively. We use roman letters for functions, signals and their arguments and greek letters mainly for operators. Also, boldface roman letters for vectors (lowcase) and matrices (capital). If $ \mathbf{M} = [m_{ij} ] $ is a matrix, its $(i,j)$-th element is also denoted as $ m_{ij} $ or as $ [\mathbf{M}]_{ij} $. Similarly, $ \mathbf{x} = [x_i] $ denotes a column vector, whose $i$-th element is denoted as $[\mathbf{x}]_i$ or simply $ x_i $. \par

\subsection{Max-plus algebra}
Max-plus arithmetic consists of the idempotent semiring $ (\mathbb{R}_\text{max}, \max, +) $, where $ \mathbb{R}_\text{max} = \mathbb{R} \cup \{-\infty\} $ is equipped with the standard maximum and sum operations, respectively.
\textit{Max-plus algebra} consists of vector operations that extend max-plus arithmetic to $\mathbb{R}_{\text{max}}^n$. They include the pointwise operations of partial ordering $ \mathbf{x} \leq \mathbf{y} $ and pointwise supremum $ \mathbf{x} \vee \mathbf{y} = [x_i \vee y_i] $, together with a class of vector transformations defined below. Max-plus algebra is isomorphic to the \textit{tropical algebra}, namely the min-plus semiring $ (\mathbb{R}_\text{min}, \min, +) $, $ \mathbb{R}_\text{min} = \mathbb{R} \cup \{\infty\} $ when extended to $\mathbb{R}_{\text{min}}^n$ in a similar fashion. Vector transformations on  $\mathbb{R}_{\text{max}}^n$ (resp. $\mathbb{R}_{\text{min}}^n$) that distribute over max-plus (resp. min-plus) vector superpositions can be represented as a max-plus $\boxplus$ (resp. min-plus $\boxplus^{'}$) product of a matrix $ \mathbf{A} \in \mathbb{R}_{\text{max}}^{m \times n} (\mathbb{R}_{\text{min}}^{m \times n}) $ with an input vector $ \mathbf{x} \in \mathbb{R}_{\text{max}}^n (\mathbb{R}_{\text{min}}^n) $:
\begin{equation}
    [\mathbf{A \boxplus x}]_{i} \triangleq \bigvee_{k = 1}^n a_{ik} + x_{k}, \; [\mathbf{A \boxplus^{'} x}]_{i} \triangleq \bigwedge_{k = 1}^n a_{ik} + x_{k}
\end{equation}
%
More details about general algebraic structures that obey those arithmetics can be found in \cite{Mara17}. In the case of a max-plus matrix equation $ \mathbf{A \boxplus x = b} $, there is a solution if and only if the vector 
\begin{equation}\label{eq:principal}\mathbf{\hat{x}} = (-\mathbf{A})^\intercal \boxplus^{'} \mathbf{b}
\end{equation}
satisfies it \cite{Cuni79,Butk10,Mara17}. We call this vector the \textit{principal solution} of the equation. It also satisfies the inequality $\mathbf{A} \boxplus \hat{\mathbf{x}} \leq \mathbf{b}$. Lastly, a vector $\mathbf{x} \in \mathbb{R}_{\text{max}}^n$ is called \textit{sparse} if it contains many $-\infty$ elements and we define its \textit{support set}, supp$(\mathbf{x})$, to be the set of positions where vector $\mathbf{x}$ has finite values, that is $ \text{supp}(\mathbf{x}) = \{i \mid x_i \neq -\infty\} $.

\subsection{Submodularity}
Let $U$ be a universe of elements. A set function $ f:2^U \to \mathbb{R} $ is called \textit{submodular}  \cite{Edmo70, Lova83} if $ \forall A \subseteq B \subseteq U, \; k \notin B $ holds:
\begin{equation}\label{eq:submod}
    f(A \cup \{k\}) - f(A) \geq f(B \cup \{k\}) - f(B).
\end{equation}
A set function $f$ is called \textit{supermodular} if $-f$ is submodular. Submodular functions occur as models of many real world evaluations in a number of fields and allow many hard combinatorial problems to be solved fast and with strong approximation guarantees \cite{KrGo14, Bach13}. It has been suggested that their importance in discrete optimization is similar to convex functions' in continuous optimization~\cite{Lova83}.\par
The following definition captures the idea of how far a given function is from being submodular and generalizes the notion of submodularity.
\begin{definition}\label{def:def43}\cite{DaKe18}
    Let $U$ be a set and $ f: 2^U \to \mathbb{R}^+ $ be an increasing, non-negative, function. The submodularity ratio of $f$ is
    \begin{equation}\label{eq:subratio}
        \gamma_{U, k}(f) \triangleq \min_{L \subseteq U, S: |S| \leq k, S \cap L = \emptyset}\frac{\sum_{x \in S}f(L \cup \{x\}) - f(L)}{f(L \cup S) - f(L)}
    \end{equation}
\end{definition}
%
%
\begin{proposition}\cite{DaKe18}
    An increasing function $ f: 2^U \to \mathbb{R} $ is submodular if and only if $\gamma_{U, k}(f) \geq 1, \; \forall \; U, k$. 
\end{proposition}
In \cite{DaKe18}, the authors used the submodularity ratio to analyze the properties of greedy algorithms in maximization problems subject to cardinality constraints and in minimum submodular cover problems, when the functions are only approximately submodular ($\gamma \in (0, 1)$). They proved that the performance of the algorithms degrade gradually as a function of $\gamma$, thus allowing guarantees for a wider variety of objective functions.    

\section{Sparse approximate solutions to max-plus equations}
     We consider the problem of finding the sparsest approximate solution to the max-plus matrix equation $ \mathbf{A \boxplus x=b}, \mathbf{A} \in \mathbb{R}^{m \times n}, \mathbf{b} \in \mathbb{R}^m $. Such a solution should i) have minimum support set $\text{supp}(\mathbf{x})$, and ii) have small enough approximation error  $\|\mathbf{b-A \boxplus x}\|_p^p$, for some $\ell_p, p < \infty$, norm. For this reason, given a prescribed constant $\epsilon$, we formulate the following optimization problem:
    \begin{equation}\label{eq:1.31}
        \begin{split}
            \displaystyle \text{arg} \min_{\mathbf{x} \in \mathbb{R}^n_{\text{max}}} & | \text{supp}(\mathbf{x}) | \\
            \text{s.t.} \; & \| \mathbf{b} - \mathbf{A} \boxplus \mathbf{x} \|_p^p \leq \epsilon, \; p < \infty, \\
            & \mathbf{A \boxplus x} \leq \mathbf{b}.
        \end{split}
    \end{equation}
    Note that we add an additional constraint $\mathbf{A \boxplus x \leq b}$, also known as the ``lateness" constraint. This constraint makes problem~\eqref{eq:1.31} more tractable; it enables the reformulation of problem~\eqref{eq:1.31} as a set optimization problem in~\eqref{eq:6}. In many applications this constraint is desirable--see~\cite{TsMa19}. However, in other situations, it might lead to less sparse solutions or higher residual error. A possible way to overcome this constraint is explored in section~\ref{ssec:SMMAE}. \par
    Even with the additional lateness constraint, problem (\ref{eq:1.31}) is very hard to solve. For example, when $ \epsilon = 0 $, solving~\eqref{eq:1.31} is an $\mathcal{NP}$-hard problem \cite{TsMa19}. Thus, we do not expect to find an efficient algorithm which solves (\ref{eq:1.31}) exactly. Instead, we will prove next there is a polynomial time algorithm which finds an approximate solution, by leveraging its supermodular properties. First, let us show that the above problem can be formed as a discrete optimization problem over a set. We follow a similar procedure to~\cite{TsMa19}, where the case $p = 1$ was examined. For the rest of this section, let $ J = \{1, \ldots, n\} $.
    
        \begin{lemma}\textbf{(Projection on the support set, $\ell_p$ case)}\label{lem:1}
        Let $T \subseteq J$, 
        \begin{equation}
            X_T = \{ \mathbf{x} \in \mathbb{R}^n_{{max}}: \text{supp}(\mathbf{x}) = T, \; \mathbf{A} \boxplus \mathbf{x} \leq \mathbf{b} \}.
        \end{equation}
        and $\mathbf{x}|_T$ be defined as $\hat{\mathbf{x}}$ inside $T$ and $-\infty$ otherwise, where $\hat{\mathbf{x}}$ is the principal solution defined in (\ref{eq:principal}). Then, it holds:
        \begin{itemize}
            \item $\mathbf{x}|_T \in X_T$.
            \item $\| \mathbf{b} - \mathbf{A} \boxplus \mathbf{x}|_T \|_p^p \leq \| \mathbf{b} - \mathbf{A} \boxplus \mathbf{x} \|_p^p \; \forall \; \mathbf{x} \in X_T$.
        \end{itemize}
    \end{lemma}
    \begin{proof}
    \; \par
    \begin{itemize}
        \item 
        It suffices to show that $ \mathbf{A} \boxplus \mathbf{x}|_T \leq \mathbf{b} $. For $j \in T$ it is $[\mathbf{x}|_T]_j = \hat{x}_j$ and for $j \in J \setminus T, [\mathbf{x}|_T]_j = -\infty \leq \hat{x}_j$. Thus,
        \begin{equation}
            \mathbf{x}|_T \leq \mathbf{\hat{x}} \iff \mathbf{A} \boxplus \mathbf{x}|_T \leq \mathbf{A} \boxplus \mathbf{\hat{x}} \implies \mathbf{A} \boxplus \mathbf{x}|_T \leq \mathbf{b}.    
        \end{equation}
        Hence, $\mathbf{x}|_T \in X_T$.
        \item
        Let $\mathbf{x} \in X_T$, then $ \mathbf{A \boxplus x } \leq \mathbf{b} \iff \mathbf{x} \leq \mathbf{\hat{x}} $, which implies (since both $\mathbf{x, x}|_T$ have $-\infty$ values outside of $T$):
        \begin{equation}
            \mathbf{x} \leq \mathbf{x}|_T \iff \mathbf{b} - \mathbf{A} \boxplus \mathbf{x}|_T \leq \mathbf{b} - \mathbf{A} \boxplus \mathbf{x}.
        \end{equation}
        Hence:
        \begin{equation}
            \| \mathbf{b} - \mathbf{A} \boxplus \mathbf{x}|_T \|_p^p = \sum_{j \in T} (\mathbf{b} - \mathbf{A} \boxplus  \mathbf{x}|_T)_j^p \leq \sum_{j \in T} (\mathbf{b} - \mathbf{A} \boxplus \mathbf{x})_j^p = \| \mathbf{b} - \mathbf{A} \boxplus \mathbf{x} \|_p^p.
        \end{equation}
    \end{itemize}
    \end{proof}
    The previous lemma informs us that we can fix the finite values of a solution of Problem (\ref{eq:1.31}) to be equal to those of the principal solution $\hat{\mathbf{x}}$. Indeed,
    \begin{proposition}\label{prop:2}
        Let $\mathbf{x}_\text{OPT}$ be an optimal solution of (\ref{eq:1.31}), then we can construct a new one with values inside the support set equal to those of the principal solution $\hat{\mathbf{x}}$.
    \end{proposition}
    \begin{proof}
        Define
        \begin{equation}
            \mathbf{z} = 
                \begin{cases} 
                    \hat{x}_j, & j \in \text{supp}(\mathbf{x}_\text{OPT}) \\
                    -\infty, & \text{otherwise}
                \end{cases},
        \end{equation}
        then $\text{supp}(\mathbf{x}_\text{OPT}) = \text{supp}(\mathbf{z}) $ and, from Lemma \ref{lem:1}, $\| \mathbf{b} - \mathbf{A} \boxplus \mathbf{z} \|_p^p \leq \| \mathbf{b} - \mathbf{A} \boxplus \mathbf{x}_\text{OPT} \|_p^p$ and $\mathbf{A} \boxplus \mathbf{z} \leq \mathbf{b}$. Thus, $\mathbf{z}$ is also an optimal solution of (\ref{eq:1.31}). 
    \end{proof}
    Therefore, the only variable that matters in Problem (\ref{eq:1.31}) is the support set. To further clarify this, let us proceed with the following definitions:  
    \begin{definition}
        Let $ T \subseteq J$ be a candidate support and let $\mathbf{A}_j $ denote the $j$-th column of $\mathbf{A}$. The \emph{error vector} $ \mathbf{e}: 2^J \to \mathbb{R}^m $ is defined as:
        \begin{equation}
            \mathbf{e}(T) = 
                \begin{cases} 
                    \mathbf{b} - \bigvee_{j \in T}(\mathbf{A}_j + \hat{x}_j), & T \neq \emptyset \\
                    \bigvee_{j \in J} \mathbf{e}(\{j\}), & T = \emptyset.
                \end{cases}
        \end{equation}
        Observe that for any $T$, it holds $ \bigvee_{j \in T}(\mathbf{A}_j + \hat{x}_j) \leq \bigvee_{j \in J}(\mathbf{A}_j + \hat{x}_j) \leq \mathbf{b} $, which means that the above vector $ \mathbf{e}(T) = (e_1(T), e_2(T), \ldots, e_m(T))^\intercal $ is always non-negative.
        We also define the corresponding error function $ E_p: 2^J \to \mathbb{R} $ as:
        \begin{equation}
            E_p(T) = \| \mathbf{e}(T) \|_p^p = \sum_{i = 1}^{m}(e_i(T))^p.
        \end{equation}
    \end{definition}
    Problem (\ref{eq:1.31}) can now be written as:
    \begin{equation}\label{eq:6}
        \begin{split}
            \displaystyle \text{arg}\min_{T \subseteq J} & \; |T| \\
            \text{s.t.} & \; E_p(T) \leq \epsilon
        \end{split}
    \end{equation}
    The main results of this section are based on the following properties of $E_p$.
    \begin{theorem}
        Error function $E_p$ is decreasing and supermodular.
    \end{theorem}
    \begin{proof}
        Regarding the monotonicity, let $ \emptyset \neq C \subseteq B \subset J $, then 
        $$ \bigvee_{j \in C}(\mathbf{A}_j + \hat{x}_j) \leq \bigvee_{j \in B}(\mathbf{A}_j + \hat{x}_j) \iff \mathbf{e}(B) \leq \mathbf{e}(C), $$ thus raising the, non-negative, components of the two vectors to the $p$-th power and adding the inequalities together yields $ E_p(B) \leq E_p(C) $. The case for $ C = \emptyset $ easily follows from the definition of $\mathbf{e}$.\par
        We employ definition (\ref{def:def43}) to help us prove the supermodularity of the function.
        Let $ S, L \subseteq U \subseteq J $, with $ |S| \leq K, S \cap L = \emptyset $ and define $ f(U) = - E_p(U) $, $ \forall \; U $. Then:
        \begin{equation*}\label{eq:8}
            \begin{split}
                 \gamma_{U, K}(f)   & = \min_{L, S} \frac{\sum_{s_k \in S} f(L \cup \{s_k\}) - f(L)}{f(L \cup S) - f(L)} = \\
                                    & = \min_{L, S} \frac{\sum_{s_k \in S}\{ -\sum_{i=1}^m[b_i - \bigvee_{j \in L \cup \{s_k\}}(A_{ij} + \hat{x}_j)]^p + \sum_{i=1}^m[b_i - \bigvee_{j \in L}(A_{ij} + \hat{x}_j)]^p \}}{-\sum_{i=1}^m[b_i - \bigvee_{j \in L \cup S}(A_{ij} + \hat{x}_j)]^p + \sum_{i=1}^m[b_i - \bigvee_{j \in L}(A_{ij} + \hat{x}_j)]^p} = \\
                                    & = \min_{L, S} \frac{\sum_{s_k \in S} \sum_{i=1}^m - [b_i - \bigvee_{j \in L \cup \{s_k\}}(A_{ij} + \hat{x}_j)]^p + [b_i - \bigvee_{j \in L}(A_{ij} + \hat{x}_j)]^p}{\sum_{i=1}^m-[b_i - \bigvee_{j \in L \cup S}(A_{ij} + \hat{x}_j)]^p + [b_i - \bigvee_{j \in L}(A_{ij} + \hat{x}_j)]^p\}}.
            \end{split}
        \end{equation*}
        Let now $I_1$ be the set: 
        \begin{equation}
            I_1 = \left\{i \in \{1, 2, 
            \ldots, m\} \mid \bigvee_{j \in L \cup S}(A_{ij} + \hat{x}_j) = \bigvee_{j \in L}(A_{ij} + \hat{x}_j)\right\}
        \end{equation}
        and for each $ s_k \in S $, we define two sets of indices:
        \begin{equation}
            I_2(s_k) =  \left\{i \in \{1, 2, \ldots, m\} \mid \bigvee_{j \in L \cup \{s_k\}}(A_{ij} + \hat{x}_j) = \bigvee_{j \in L \cup S}(A_{ij} + \hat{x}_j) > \bigvee_{j \in L}(A_{ij} + \hat{x}_j)\right\}
        \end{equation}
        and:
        \begin{equation}
            I_3(s_k) = \left\{i \in \{1, 2, \ldots, m\} \mid \bigvee_{j \in L \cup S}(A_{ij} + \hat{x}_j) > \bigvee_{j \in L \cup \{s_k\}}(A_{ij} + \hat{x}_j)  > \bigvee_{j \in L}(A_{ij} + \hat{x}_j)\right\}.
        \end{equation}
        Then, if 
        \begin{equation}
            \Sigma_1 = \sum_{s_k \in S}\sum_{i \in I_1, I_2(s_k)} -[b_i - \bigvee_{j \in L \cup \{s_k\}}(A_{ij} + \hat{x}_j)]^p + [b_i - \bigvee_{j \in L}(A_{ij} + \hat{x}_j)]^p,
        \end{equation}
        the ratio becomes:
        \begin{equation*}
            \begin{split}
                \gamma_{U, K}(f) & = \min_{L, S} \frac{\Sigma_1 + \sum_{s_k \in S}\sum_{i \in I_3(s_k)} -[b_i - \bigvee_{j \in L \cup \{s_k\}}(A_{ij} + \hat{x}_j)]^p + [b_i - \bigvee_{j \in L}(A_{ij} + \hat{x}_j)]^p}{\Sigma_1}\\
                & \geq 1, \; \forall \; U, K.
            \end{split}
        \end{equation*}
        meaning that $f$ is submodular or, equivalently, $E_p = -f$ is supermodular.
    \end{proof}
    \begin{algorithm}
    \SetAlgoLined
    \SetKwInOut{Input}{Input}
    \Input{$\mathbf{A, b}$}
    
     Compute $ \mathbf{\hat{x}} = (-\mathbf{A})^\intercal \boxplus^{'} \mathbf{b} $\\
     \uIf{$E_p(J) > \epsilon$}{\textbf{return} Infeasible}
     Set $ T_0 = \emptyset, k = 0 $\\
     \While{$ E_p(T_k) > \epsilon $}{
      $ j = \text{arg}\min_{s \in J \setminus T_k}E_p(T_k \cup \{s\}) $\\
      $ T_{k+1} = T_k \cup \{j\} $\\
      $ k = k + 1 $
     }
     $ x_j = \hat{x}_j, j \in T_k $ and $ x_j = -\infty $, otherwise\\
     \textbf{return} $ \mathbf{x}, T_k $
     \caption{Approximate solution of problem (\ref{eq:1.31})}
     \label{alg:greedy}
    \end{algorithm}
    Setting $ \tilde{E}_p(T) = \max(E_p(T), \epsilon) $ \footnote{The new, truncated, error function remains supermodular; see \cite{KrGo14}.} and leveraging the previous theorem, we are able to formulate problem (\ref{eq:6}), and thus the initial one (\ref{eq:1.31}), as a cardinality minimization problem subject to a supermodular equality constraint \cite{Wols82}, which allows us to approximately solve it by the greedy Algorithm \ref{alg:greedy}. The calculation of the principal solution requires $\mathcal{O}(nm)$ time and the greedy selection of the support set of the solution costs $\mathcal{O}(n^2)$ time. We call the solutions of problem (\ref{eq:1.31}) \textit{Sparse Greatest Lower Estimates (SGLE)} of $\mathbf{b}$. Regarding the approximation ratio between the optimal solution and the output of Algorithm \ref{alg:greedy}, the following proposition holds.
   \begin{proposition}\label{prop:ratio}
        Let $\mathbf{x}$ be the output of Algorithm \ref{alg:greedy} after $k > 0$ iterations of the inner while loop and $T_k$ the respective support set. Then, if $T^*$ is the support set of the optimal solution of (\ref{eq:1.31}), the following inequality holds:
        \begin{equation}
            \frac{|T_k|}{|T^*|} \leq 1 + \log\left(\frac{m\Delta^p - \epsilon}{E_p(T_{k-1}) - \epsilon}\right),    
        \end{equation}
        where $ \Delta = \bigvee_{i, j}(b_i - A_{ij} - \hat{x}_j) $.
    \end{proposition}
    \begin{proof}
        From \cite{Wols82}, the following bound holds for the cardinality minimization problem subject to a supermodular and decreasing constraint, defined as function $f: 2^J \to \mathbb{R} $, by the greedy algorithm:
        \begin{equation}
            \frac{|T_k|}{|T^*|} \leq 1 + \log\left(\frac{f(\emptyset) - f(J)}{f(T_{k-1}) - f(J)}\right)
        \end{equation}
        For our problem, it is $f = \tilde{E}_p$. Observe now that, since $ k > 0 $, $ \tilde{E}_p(\emptyset) = E_p(\emptyset) \leq m\Delta^p $, $ 0 \leq \tilde{E}_p(J) = \epsilon $ and $ \tilde{E}_p(T_{k-1}) > \epsilon $. Therefore, the result follows. 
    \end{proof}
    The ratio warn us to expect less optimal and, thus, less sparse vectors when increasing the norm $p$ that we use to measure the approximation. It also hints towards an inapproximability result when $p \to \infty$, which is formalised next.
    
    \subsection{Sparse vectors with minimum \texorpdfstring{$\ell_\infty$}{l-infinity} errors}
    \label{ssec:SMMAE}
    Although in some settings the $\mathbf{A \boxplus x} \leq \mathbf{b} $ constraint is needed~\cite{TsMa19}, in other cases it could disqualify potentially sparsest vectors from consideration. Omitting the constraint, on the other hand, makes it unclear how to search for minimum error solutions for any $\ell_p \; (p < \infty$) norm. For instance, it has recently been reported that it is $\mathcal{NP}$-hard to determine if a given point is a local minimum for the $\ell_2$ norm \cite{Hook19}. For that reason, we shift our attention to the case of $ p = \infty $. It is well known \cite{Cuni79, Butk10} that problem $ \min_{\mathbf{x} \in \mathbb{R}^n_{\text{max}}} \| \mathbf{b} - \mathbf{A} \boxplus \mathbf{x} \|_\infty $ has a closed form solution; it can be calculated in $ \mathcal{O}(nm) $ time by adding to the principal solution element-wise the half of its $\ell_\infty$ error. Note that this new vector does not necessarily satisfy $ \mathbf{A \boxplus x} \leq \mathbf{b} $, so it shows a way to overcome the aforementioned limitation. \par
    First, let us demonstrate that problem (\ref{eq:1.31}), when considering the $\ell_\infty$ norm, becomes harder than before and non-approximable by the greedy Algorithm \ref{alg:greedy}. Hence, consider now the following optimization problem:
    \begin{equation}\label{eq:linf_problem}
        \begin{split}
            \displaystyle \text{arg} \min_{\mathbf{x} \in \mathbb{R}^n_{\text{max}}} & | \text{supp}(\mathbf{x}) | \\
            \text{s.t.} \; & \| \mathbf{b} - \mathbf{A} \boxplus \mathbf{x} \|_\infty \leq \epsilon.
        \end{split}
    \end{equation}
    Thanks to a similar construction as in the previous section, this problem can be recast as a set-search problem.
    \begin{lemma}\textbf{(Projection on the support set, $\ell_\infty$ case)}\label{lem:2}
        Let $T \subseteq J$, $\mathbf{x}|_T$ defined as $\hat{\mathbf{x}}$ inside $T$ and $-\infty$ otherwise and $ \mathbf{x}^* = \mathbf{x}|_T + \frac{\| \mathbf{b} - \mathbf{A} \boxplus \mathbf{x}|_T \|_\infty}{2} $. Then $ \forall \; \mathbf{z} \in \mathbb{R}^n_\text{max} $ with $ \text{supp}(\mathbf{z}) = T $, it holds:
        \begin{equation}
            \| \mathbf{b} - \mathbf{A} \boxplus \mathbf{z} \|_\infty \geq \| \mathbf{b} - \mathbf{A} \boxplus \mathbf{x^*}  \|_\infty = \frac{\| \mathbf{b} - \mathbf{A} \boxplus \mathbf{x}|_T \|_\infty}{2}.
        \end{equation}
    \end{lemma}
    \begin{proof}(Sketch)
        By fixing the support set of the considered vectors equal to $T$, equivalently we omit the columns and indices of $\mathbf{A}$ and $\mathbf{x}$, respectively, that do not belong in $T$ (since they will not be considered at the evaluation of the maximum). By doing so, we get a new equation with same vector $\mathbf{b}$ and restricted $\mathbf{A, x}$. The vector $\mathbf{x}^*$ that minimizes the $\ell_\infty$ error of this equation is obtained from its principal solution plus the half of its $\ell_\infty$ error. But now observe that the new principal solution shares the same values with the original principal solution (follows from Lemma \ref{lem:1}) inside $T$, which is exactly vector $\mathbf{x}|_T$. Extending $\mathbf{x}^*$ back to $\mathbb{R}^n_\text{max}$ yields the result.
    \end{proof}
    So, a similar result to Proposition \ref{prop:2} holds.
    \begin{proposition}
        Let $\mathbf{x}_\text{OPT}$ be an optimal solution of (\ref{eq:linf_problem}), then we can construct a new one with values inside the support set equal to those of the principal solution $\hat{\mathbf{x}}$ plus the half of its $\ell_\infty$ error.
    \end{proposition}
    By defining $ E_\infty(T) = \frac{\| \mathbf{b} - \mathbf{A} \boxplus \mathbf{x}|_T \|_\infty}{2} $, (\ref{eq:linf_problem}) becomes:
    \begin{equation}\label{eq:infty_set}
        \begin{split}
            \displaystyle \text{arg}\min_{T \subseteq J} & \; |T| \\
            \text{s.t.} & \; E_\infty(T) \leq \epsilon
        \end{split}
    \end{equation}
    Unfortunately this problem does not admit an approximate solution by the greedy Algorithm \ref{alg:greedy} (to be precise, the modified version of Algorithm \ref{alg:greedy} when $E_p$ becomes $E_\infty$), as its error function, although decreasing, is not supermodular. The following example also reveals that the submodularity ratio (\ref{eq:subratio}) of $E_\infty$ is $0$. Therefore, it is not even approximately supermodular and a solution by Algorithm \ref{alg:greedy} can be arbitrarily bad~\cite{DaKe18}.
    \begin{example}
        Let $A = 
        \begin{pmatrix}
            0 & 5 & 2 \\
            4 & 1 & 0 \\
            0 & 1 & 0
        \end{pmatrix}, \mathbf{b} = 
        \begin{pmatrix}
            3 \\ 1 \\ 0
        \end{pmatrix}$, then principal solution $\hat{\mathbf{x}}$ is: 
        $$\hat{\mathbf{x}} = 
        \begin{pmatrix}
            0 & -4 & 0 \\
            -5 & -1 & -1 \\
            -2 & 0 & 0
        \end{pmatrix} \boxplus' 
        \begin{pmatrix}
            3 \\ 1 \\ 0
        \end{pmatrix} = 
        \begin{pmatrix}
            -3 \\ -2 \\ 0
        \end{pmatrix}.$$ We calculate now the error function on different sets:
        \begin{itemize}
            \item When $T = \{3\}$, then  $\hat{\mathbf{x}}|_{\{3\}} = \begin{pmatrix} -\infty, -\infty, 0 \end{pmatrix}^\intercal $ and \newline $ E_\infty(\{3\}) = \frac{1}{2} \| \mathbf{b} - \bigvee_{j \in \{3\}}(\mathbf{A}_j + \hat{x}|_{\{3\},j}) \|_\infty = \frac{1}{2} \|
                \begin{pmatrix}
                    3 \\ 1 \\ 0
                \end{pmatrix} - 
                \begin{pmatrix}
                    2 \\ 0 \\ 0
                \end{pmatrix} \|_\infty = \frac{1}{2} $.
            \item Likewise, when $T = \{1, 3\}$, $ E_\infty(\{1, 3\}) = \frac{1}{2} \|
                \begin{pmatrix}
                    3 \\ 1 \\ 0
                \end{pmatrix} -
                \begin{pmatrix}
                    -3 \\ 1 \\ -3
                \end{pmatrix}
                \bigvee
                \begin{pmatrix}
                    2 \\ 0 \\ 0
                \end{pmatrix} \|_\infty = \frac{1}{2} $.
            \item $T = \{2, 3\}$, $ E_\infty(\{2, 3\}) = \frac{1}{2} \| 
                \begin{pmatrix}
                    3 \\ 1 \\ 0
                \end{pmatrix} -
                \begin{pmatrix}
                    3 \\ -1 \\ -1
                \end{pmatrix}
                \bigvee
                \begin{pmatrix}
                    2 \\ 0 \\ 0
                \end{pmatrix} \|_\infty = \frac{1}{2} $.
            \item $T = \{1, 2, 3\}$, $ E_\infty(\{1, 2, 3\}) = \frac{1}{2} \|
                \begin{pmatrix}
                    3 \\ 1 \\ 0
                \end{pmatrix} - 
                \begin{pmatrix}
                    -3 \\ 1 \\ -3
                \end{pmatrix}
                \bigvee
                \begin{pmatrix}
                    3 \\ -1 \\ -1
                \end{pmatrix}
                \bigvee
                \begin{pmatrix}
                    2 \\ 0 \\ 0
                \end{pmatrix} 
                \|_\infty = 0$.
        \end{itemize}
        Let now $f = - E_\infty, L = \{3\}, S = \{1, 2\} $, then, by (\ref{eq:subratio}), we have:
        \begin{equation}
            \frac{ f(\{3\} \cup \{1\} - f(\{3\}) + f(\{3\} \cup \{2\}) - f(\{3\}) }{ f(\{3\} \cup \{1, 2\}) - f(\{3\}) } = \frac{-1/2 + 1/2 - 1/2 + 1/2}{0 + 1/2} = 0,
        \end{equation}
        meaning that f has submodularity ratio $0$ or $E_\infty$ is not even approximately supermodular.
    \end{example}
    
    Although the previous discussion denies from problem (\ref{eq:linf_problem}) a greedy solution with any guarantees, we propose next a practical alternative to get a sparse enough vector.
    We first obtain a sparse vector $\mathbf{x}_{p, \epsilon}$ by solving problem~\eqref{eq:1.31}. Then, we add to the vector element-wise half of its $\ell_\infty$ error $\| \mathbf{b} - \mathbf{A} \boxplus \mathbf{x}_{p, \epsilon} \|_\infty/2$. Interestingly, this new solution minimizes the $\ell_{\infty}$ error among all vectors with the same support, as formalized in the following result.
     \begin{proposition}\label{prop:SMMAE}
        Let $\mathbf{x}_{SMMAE} \in \mathbb{R}^n_\text{max}$ be defined as:
        \begin{equation}
            \mathbf{x}_{SMMAE} = \mathbf{x}_{p, \epsilon} + \frac{\| \mathbf{b} - \mathbf{A} \boxplus \mathbf{x}_{p, \epsilon} \|_\infty}{2},
        \end{equation}
        where $\mathbf{x}_{p, \epsilon}$ is a solution of problem (\ref{eq:1.31}) with fixed $(p, \epsilon)$. Then $ \forall \; \mathbf{z} \in \mathbb{R}^n_\text{max} $ with $ \text{supp}(\mathbf{z}) = \text{supp}(\mathbf{x}_{p, \epsilon}) $, it holds
        \begin{equation}
            \| \mathbf{b} - \mathbf{A} \boxplus \mathbf{z} \|_\infty \geq \| \mathbf{b} - \mathbf{A} \boxplus \mathbf{x}_{SMMAE}  \|_\infty = \frac{\| \mathbf{b} - \mathbf{A} \boxplus \mathbf{x}_{p, \epsilon} \|_\infty}{2} 
        \end{equation}
        and, also,
        \begin{equation}\label{eq:12}
            \| \mathbf{b} - \mathbf{A} \boxplus \mathbf{x}_{SMMAE} \|_\infty \leq \frac{\sqrt[p]{\epsilon}}{2}.
        \end{equation}
    \end{proposition}
    
    \begin{proof}
        Observe that $\mathbf{x}_{p, \epsilon}$ is equal to the principal solution $\hat{\mathbf{x}}$ inside $\text{supp}(\mathbf{x}_{p, \epsilon})$. So the first inequality holds from Lemma \ref{lem:2}. Regarding the second one, we have: 
        \begin{equation}\label{eq:smmaeproof1}
            \| \mathbf{b} - \mathbf{A \boxplus x}_{SMMAE} \|_\infty = \frac{\| \mathbf{b} - \mathbf{A \boxplus x}_{p, \epsilon} \|_\infty}{2} = \frac{\bigvee_i(b_i - [\mathbf{A \boxplus x}_{p, \epsilon}]_i)}{2}.
        \end{equation}
        But, notice that:
        \begin{equation}\label{eq:smmaeproof2}
            (\bigvee_i b_i - [\mathbf{A \boxplus x}_{p, \epsilon}]_i)^p = \bigvee_i(b_i - [\mathbf{A \boxplus x}_{p, \epsilon}]_i)^p \leq \sum_i(b_i - [\mathbf{A \boxplus x}_{p, \epsilon}]_i)^p \leq \epsilon,
        \end{equation}
        so
        \begin{equation}\label{eq:smmaeproof3}
            \bigvee_i(b_i - [\mathbf{A \boxplus x}_{p, \epsilon}]_i) \leq \sqrt[p]{\epsilon}
        \end{equation}
        and the result follows from (\ref{eq:smmaeproof1}). Note that the bound tightens, as $p$ increases.
    \end{proof}
    
    The above method provides sparse vectors that are approximate solutions of the equation with respect to the $\ell_\infty$ norm without the need of the lateness constraint. After computing $\mathbf{x}_{p, \epsilon}$, $\mathbf{x}_{SMMAE}$ requires $ \mathcal{O}(m|\text{supp}(\mathbf{x}_{p, \epsilon})| + |\text{supp}(\mathbf{x}_{p, \epsilon})|) = \mathcal{O}((m+1)|\text{supp}(\mathbf{x}_{p, \epsilon})|) $ time. We call $\mathbf{x}_{\text{SMMAE}}$ \textit{Sparse Minimum Max Absolute Error (SMMAE)} estimate of $\mathbf{b}$. 
    
    Next, we provide an experiment on randomly generated data to assess the effectiveness of the proposed method in solving the $\ell_\infty$ problem (\ref{eq:linf_problem}).
    \begin{example}
        Input data consist of $100$ random pair of matrices $ \mathbf{A} \in \mathbb{R}^{1000 \times 1000}, \mathbf{b} \in \mathbb{R}^{1000 \times 1} $ where each value of $\mathbf{A}$ is sampled from a normal distribution $\mathcal{N}(0, 2^2)$ and each value of $\mathbf{b}$ from a standard one $\mathcal{N}(0, 1)$. \par
        We organise the experiment as follows: for each pair of $\mathbf{A, b}$, we solve, first, problem (\ref{eq:1.31}) with $p=150$ and $\epsilon=(2 \cdot 2.5)^{150}$ to acquire a sparse vector that is an approximate solution with respect to the $\ell_{150}$ norm and then add to it half of its $\ell_\infty$ error, obtaining this way a sparse vector that has $\ell_\infty$ error less than $2.5$ (see Proposition \ref{prop:SMMAE}). We choose a high order norm so that $\ell_\infty$ is close to its theoretical bound $2.5$. Afterwards, we solve directly the $\ell_\infty$ problem (\ref{eq:linf_problem}) with greedy Algorithm \ref{alg:greedy} (with a change of error functions i.e. $E_p$ becomes $E_\infty$) for $\epsilon = 2.5$ and compare the cardinalities of the support set of the solutions produced from the two methods. \par
        The heuristic method has a median cardinality of $30$ as opposed to $33$ for the greedy approach, which verifies the soundness of the proposed method. Although the benefit seems small, Fig. \ref{fig:comp} reveals that the greedy approximation can have unnecessary large support set (observe the spikes on the Greedy graph) and the difference between the two methods can be arbitrarily big.

        \begin{figure}
            \centering
            \includegraphics[width=0.75\textwidth, height=0.4\textheight]{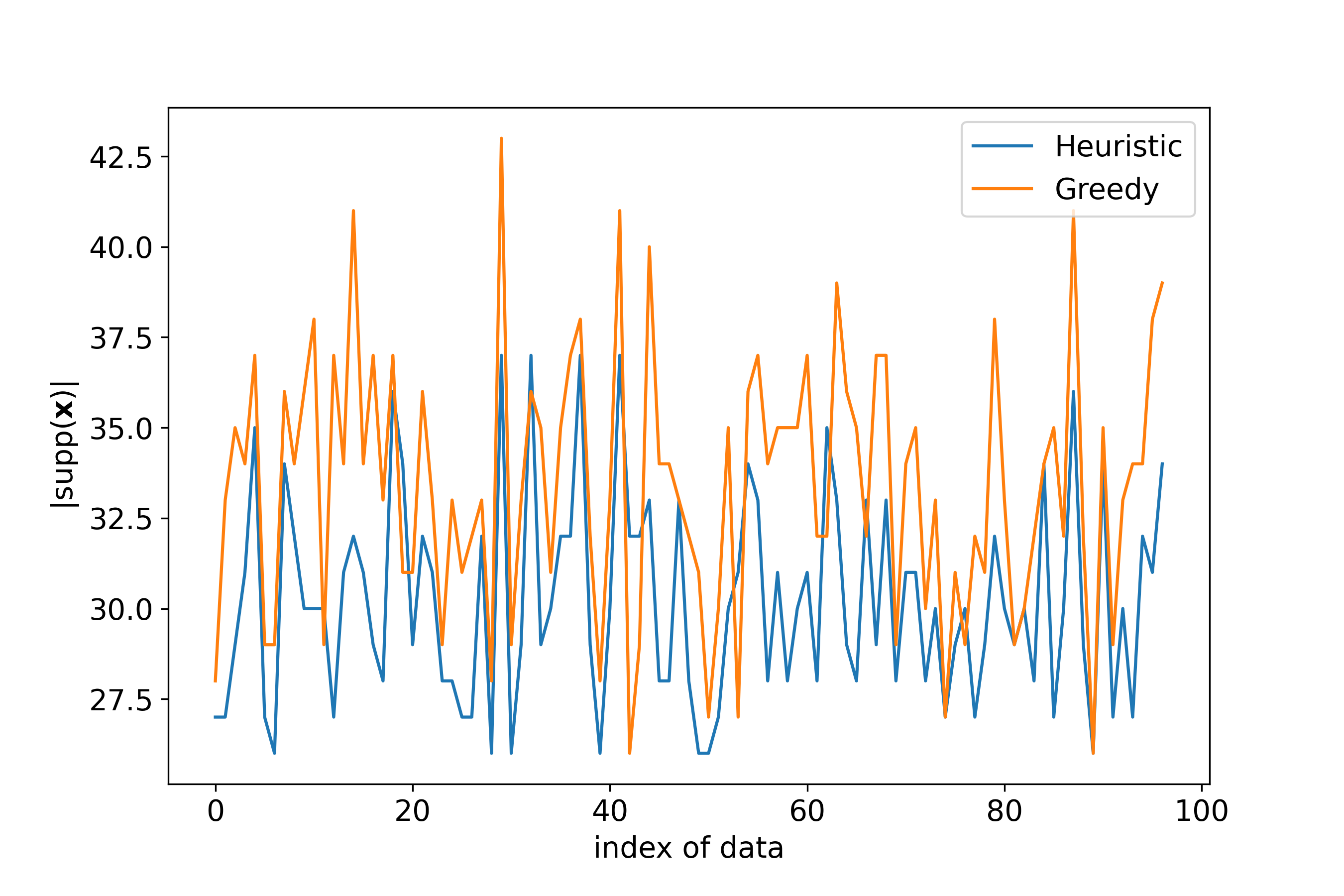}
            \caption{The cardinality of the support set obtained from a greedy solution of (\ref{eq:linf_problem}) and the heuristic approach proposed in Proposition \ref{prop:SMMAE}. Shown for $100$ different pairs of input data $\mathbf{A, b}$. Best viewed in color.}
            \label{fig:comp}
        \end{figure}
    \end{example}

\section{Applications in convex regression}
    In this section, we are interested in approximating a convex function by a piecewise-linear one. We call this the \textit{Tropical Regression problem}. It is well known that any convex function can be expressed as the pointwise supremum of a, potentially infinite, family of affine hyperplanes, using the Legendre-Fenchel conjugate (a.k.a. slope transform)~\cite{Rock70, BoVa04, HeMa97}. Our goal is to  approximate the convex function with as few hyperplanes as possible. We show next how the sparse framework we introduced addresses this problem. \par
    Let $ (\mathbf{x}_i, f_i) \in \mathbb{R}^{n+1}, i = 1, \ldots, m, $ be a set of (possibly noisy) data sampled from a convex function $f$ and $\{\mathbf{a}_k\}_{k=1}^{K}$ be a set of slope vectors; for example, this could be some integer multiples of a slope step inside a fixed $n$-dimensional interval or the numerical gradients of the data. Given the data and the slopes, our goal is to compute a PWL (piecewise-linear) function $p$:
    \begin{equation}\label{eq:10}
        p(\mathbf{x}) = \bigvee_{k = 1}^K \mathbf{a}^{\intercal}_k\mathbf{x} + b_k,
    \end{equation}
    that satisfies $ f_i = p(x_i) + \text{error}, \forall i $. Ideally, this regression problem can be formulated as the following max-plus matrix equation:
    \begin{equation}\label{eq:11}
        \underbrace{\begin{pmatrix}
            \mathbf{a}^\intercal_1 \mathbf{x}_1 & \mathbf{a}^\intercal_2 \mathbf{x}_1 & \ldots & \mathbf{a}^\intercal_K \mathbf{x}_1 \\
            . & . & . & . \\
            . & . & . & . \\
            \mathbf{a}^\intercal_1 \mathbf{x}_m & \mathbf{a}^\intercal_2 \mathbf{x}_m & \ldots & \mathbf{a}^\intercal_K \mathbf{x}_m
        \end{pmatrix}}_{\mathbf{A}} \boxplus 
        \underbrace{\begin{pmatrix}
            b_1 \\ b_2 \\ . \\ . \\ b_K
        \end{pmatrix}}_{\mathbf{x}} = 
        \underbrace{\begin{pmatrix}
            f_1 \\ . \\ . \\ f_m
        \end{pmatrix}}_{\mathbf{b}}
    \end{equation}
    Observe that by taking $b_k=-\infty$, the hyperplane $\mathbf{a}^{\intercal}_k\mathbf{x} + b_k$ is neglected in the maximum. Hence, sparsity leads to using less affine regions.
    We can formulate problem (\ref{eq:1.31}) for the above matrices for any desired $(\epsilon, p)$. If a solution exists, then it produces intercepts $b_k$ that ensure that the $\ell_p $ approximation error is less than $\epsilon$ and, at the same time, the resulting tropical polynomial contains the approximately minimum number of affine regions needed to approximate $f$. Except for the previous SGLEs, we are also able to get the SMMAE estimates of $f$ by adding to the result half of its $\ell_\infty$ error, as explained in section \ref{ssec:SMMAE}. Coming with $\ell_\infty$ guarantees, those estimates are useful especially when the approximation is being used as a surrogate of the original function in an optimization problem, as the difference between the $2$ minima can be bounded. \par
    First, we calculate matrix $\mathbf{A}$ in $\mathcal{O}(Knm)$. Solving, now, problem (\ref{eq:1.31}) for equation (\ref{eq:11}) requires the computation of its principal solution in $\mathcal{O}(Km)$ time and then employing the greedy algorithm to find the intercepts $b_k$ with complexity $\mathcal{O}(K^2)$, meaning a total complexity of $\mathcal{O}(K^2 + K(n+1)m)$. Computing the SMMAE estimate, as well, requires an extra $\mathcal{O}(Km)$. Next, we demonstrate the effectiveness of our method via numerical examples. \par
    
    \begin{example}\label{ex:1}
        Consider $100$ pairs of noiseless data $(x_i, y_i)$, where $ x_i $ are evenly spaced numbers sampled from the interval $ [-2, 2] $ and $ y_i = f(x_i) $, where f is the convex function:
        \begin{equation}
            f(x) = \max(-6x-6, \frac{x}{2}, \frac{x^5}{5} + \frac{x}{2}).
        \end{equation}
        We wish to fit the following max-plus tropical polynomial curve, where we fix the candidate slopes to be the set of all $ k \in [-20, 20] $, with a step size of $0.125$:
        \begin{equation}
            p(x) = \bigvee_{k = -20}^{20} kx + b_k,
        \end{equation}
        so the corresponding equations become:
        \begin{equation}\label{eq:ex1matrix}
            \begin{pmatrix}
                -20 x_1 & -19.875 x_1 & 19.75 x_1 & \ldots & 20 x_1 \\
                . & . & . & . & . \\
                . & . & . & . & . \\
                -20 x_{100} & -19.875 x_{100} & -19.75 x_{100} & \ldots & 20 x_{100}
            \end{pmatrix} \boxplus 
            \begin{pmatrix}
                b_{-20} \\ b_{-19.875} \\ b_{-19.75} \\ . \\ . \\ b_{20}
            \end{pmatrix} = 
            \begin{pmatrix}
                f_1 \\ . \\ . \\ f_{100}
            \end{pmatrix}
        \end{equation}
    We solve problem (\ref{eq:1.31}) for the above matrices and for a variety of different pairs of error threshold and norm order to obtain sparse greatest lower estimates (SGLE) and then add to these solutions the half of their $\ell_\infty$ error in order to get the corresponding sparse minimum max absolute error (SMMAE) estimates. In order to provide a clarifying comparison between solutions obtained with different $p$ norms, for each experiment we set the error threshold $\epsilon$ to be $\theta^p$, where $\theta$ is varied. We present the resulting SGLEs and SMMAEs in Tables \ref{tab:1}, \ref{tab:2} and \ref{tab:3}, \ref{tab:4}, respectively. Notice that the SMMAE estimates have exactly half the $\ell_\infty$ error of the respective SGLEs, as expected by Proposition \ref{prop:SMMAE}. Also, observe in Tables \ref{tab:2} and \ref{tab:4} the effect of increasing the norm order $p$ to the resulting support set (it is increased as suggested by Proposition \ref{prop:ratio}).  See Figure \ref{fig:1} for the best PWL approximations of $f$.
    
    \begin{table}
        \centering
        \begin{tabular}{|l|ccc||ccc|}
            \cline{2-7}
            \multicolumn{1}{}{}
            &
            \multicolumn{3}{|c||}{$p = 1$}
            &
            \multicolumn{3}{c|}{$p = 2$}\\
            \hline
            $\theta$ & erro${r}_{RMS}$ & erro$\text{r}_\infty$ & $|$supp$|$ & erro${r}_{RMS}$ & erro$\text{r}_\infty$ & $|$supp$|$\\
            \hline
            $0.15$ & $0.0038$ & $0.0226$ & $15$ & $0.0131$ & $0.0532$ & $10$\\
            $0.25$ & $0.0057$ & $0.0376$ & $13$ & $0.0230$ & $0.0932$ & $7$\\
            $0.5$ & $0.0120$ & $0.0697$ & $11$ & $0.0436$ & $0.2354$ & $6$\\
            $1$ & $0.0202$ & $0.1071$ & $8$ & $0.0628$ & $0.2354$ & $5$\\
            $2$ & $0.0491$ & $0.2794$ & $6$ & $0.1525$ & $1.0099$ & $4$\\
            $3$ & $0.0615$ & $0.2794$ & $5$ & $0.2521$ & $1.0099$ & $3$\\
            $4$ & $0.0615$ & $0.2794$ & $5$ & $0.2521$ & $1.0099$ & $3$\\
            $10$ & $0.1628$ & $1.0824$ & $4$ & $0.2521$ & $1.0099$ & $3$\\
            $15$ & $0.2529$ & $1.0824$ & $3$ & $1.4335$ & $6.4000$ & $2$\\
            $30$ & $0.2529$ & $1.0824$ & $3$ & $2.5800$ & $7.0000$ & $1$\\
            \hline
        \end{tabular}
        \caption{$\ell_1$ and $\ell_2$ SGLEs obtained from solving problem (\ref{eq:1.31}) for equation (\ref{eq:ex1matrix}), with $ p = 1, 2$, respectively, and error threshold $\epsilon^p$.  We report the \textbf{R}oot \textbf{M}ean \textbf{S}quared and Maximum Absolute errors, along with the cardinality of the support set of the solution (the number of affine regions of the resulting tropical polynomial).}
        \label{tab:1}
    \end{table}
    
    \begin{table}
        \centering
        \begin{tabular}{|l|ccc||ccc|}
            \cline{2-7}
            \multicolumn{1}{}{}
            &
            \multicolumn{3}{|c||}{$p = 5$}
            &
            \multicolumn{3}{c|}{$p = 150$}\\
            \hline
            $\theta$ & erro${r}_{RMS}$ & erro$\text{r}_\infty$ & $|$supp$|$ & erro${r}_{RMS}$ & erro$\text{r}_\infty$ & $|$supp$|$\\
            \hline
            $0.15$ & $0.0228$ & $0.0932$ & $7$ & $0.0458$ & $0.1313$ & $18$\\
            $0.25$ & $0.0228$ & $0.0932$ & $7$ & $0.0647$ & $0.2322$ & $16$\\
            $0.5$ & $0.0648$ & $0.2497$ & $5$ & $0.1699$ & $0.3867$ & $13$\\
            $1$ & $0.1430$ & $0.9392$ & $4$ & $0.2735$ & $0.8685$ & $10$\\
            $2$ & $0.2530$ & $0.9392$ & $3$ & $0.6084$ & $1.8232$ & $7$\\
            $3$ & $0.2530$ & $0.9392$ & $3$ & $0.9615$ & $2.8788$ & $5$\\
            $4$ & $0.2530$ & $0.9392$ & $3$ & $1.1120$ & $3.6444$ & $4$\\
            $10$ & $1.4335$ & $6.4000$ & $2$ & $2.6230$ & $6.8636$ & $1$\\
            $15$ & $2.5800$ & $7.0000$ & $1$ & $2.6230$ & $6.8636$ & $1$\\
            $30$ & $2.5800$ & $7.0000$ & $1$ & $2.6230$ & $6.8636$ & $1$\\
            \hline
        \end{tabular}
        \caption{$\ell_5$ and $\ell_{150}$ SGLEs for a number of different error thresholds. Same metrics reported, as in Table \ref{tab:1}.}
        \label{tab:2}
    \end{table}
    
    \begin{table}
        \centering
        \begin{tabular}{|l|ccc||ccc|}
            \cline{2-7}
            \multicolumn{1}{}{}
            &
            \multicolumn{3}{|c||}{$p = 1$}
            &
            \multicolumn{3}{c|}{$p = 2$}\\
            \hline
            $\theta$ & erro${r}_{RMS}$ & erro$\text{r}_\infty$ & $|$supp$|$ & erro${r}_{RMS}$ & erro$\text{r}_\infty$ & $|$supp$|$\\
            \hline
            $0.15$ & $0.0105$ & $0.0113$ & $15$ & $0.0243$ & $0.0266$ & $10$\\
            $0.25$ & $0.0176$ & $0.0189$ & $13$ & $0.0415$ & $0.0466$ & $7$\\
            $0.5$ & $0.0328$ & $0.0349$ & $11$ & $0.1080$ & $0.1177$ & $6$\\
            $1$ & $0.0486$ & $0.0535$ & $8$ & $0.1053$ & $0.1177$ & $5$\\
            $2$ & $0.1297$ & $0.1398$ & $6$ & $0.4733$ & $0.5049$ & $4$\\
            $3$ & $0.1252$ & $0.1397$ & $5$ & $0.4552$ & $0.5049$ & $3$\\
            $4$ & $0.1252$ & $0.1398$ & $5$ & $0.4552$ & $0.5049$ & $3$\\
            $10$ & $0.5096$ & $0.5412$ & $4$ & $0.4552$ & $0.5049$ & $3$\\
            $15$ & $0.4879$ & $0.5412$ & $3$ & $2.9508$ & $3.2000$ & $2$\\
            $30$ & $0.4879$ & $0.5412$ & $3$ & $2.8645$ & $3.5000$ & $1$\\
            \hline
        \end{tabular}
        \caption{$\ell_1$ and $\ell_2$ SMMAE estimates for a number of different error thresholds. Same metrics reported, as in Table \ref{tab:1}.}
        \label{tab:3}
    \end{table}
    
    \begin{table}
        \centering
        \begin{tabular}{|l|ccc||ccc|}
            \cline{2-7}
            \multicolumn{1}{}{}
            &
            \multicolumn{3}{|c||}{$p = 5$}
            &
            \multicolumn{3}{c|}{$p = 150$}\\
            \hline
            $\theta$ & erro${r}_{RMS}$ & erro$\text{r}_\infty$ & $|$supp$|$ & erro${r}_{RMS}$ & erro$\text{r}_\infty$ & $|$supp$|$\\
            \hline
            $0.15$ & $0.0414$ & $0.0466$ & $7$ & $0.0545$ & $0.657$ & $18$\\
            $0.25$ & $0.0414$ & $0.0466$ & $7$ & $0.0945$ & $0.1161$ & $16$\\
            $0.5$ & $0.1119$ & $0.1248$ & $5$ & $0.1265$ & $0.1933$ & $13$\\
            $1$ & $0.4385$ & $0.4696$ & $4$ & $0.3093$ & $0.4342$ & $10$\\
            $2$ & $0.4245$ & $0.4696$ & $3$ & $0.7243$ & $0.9116$ & $7$\\
            $3$ & $0.4245$ & $0.4696$ & $3$ & $1.1728$ & $1.4394$ & $5$\\
            $4$ & $0.4245$ & $0.4696$ & $3$ & $1.4588$ & $1.8222$ & $4$\\
            $10$ & $2.9508$ & $3.2000$ & $2$ & $2.7175$ & $3.4318$ & $1$\\
            $15$ & $2.8645$ & $3.5000$ & $1$ & $2.7175$ & $3.4318$ & $1$\\
            $30$ & $2.8645$ & $3.5000$ & $1$ & $2.7175$ & $3.4318$ & $1$\\
            \hline
        \end{tabular}
        \caption{$l_5$ and $l_{150}$ SMMAE estimates for a number of different error thresholds. Same metrics reported, as in Table \ref{tab:1}.}
        \label{tab:4}
    \end{table}
    
    \begin{figure}
        \centering
        \begin{subfigure}{0.45\textwidth}
            \includegraphics[width=\textwidth]{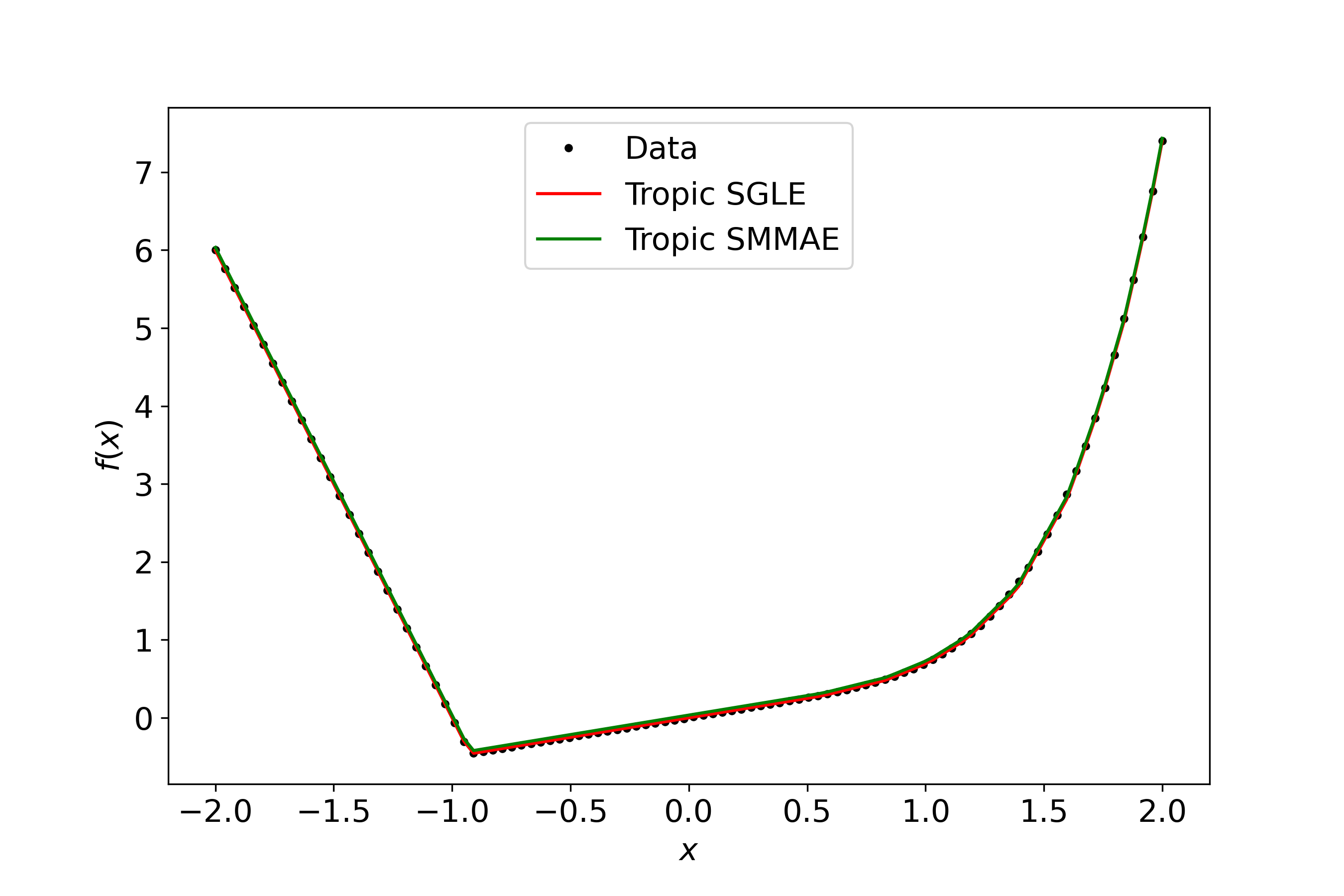}
            \caption{$K = 11, \epsilon = 0.5, p = 1$}
        \end{subfigure}
        \begin{subfigure}{0.45\textwidth}
            \includegraphics[width=\textwidth]{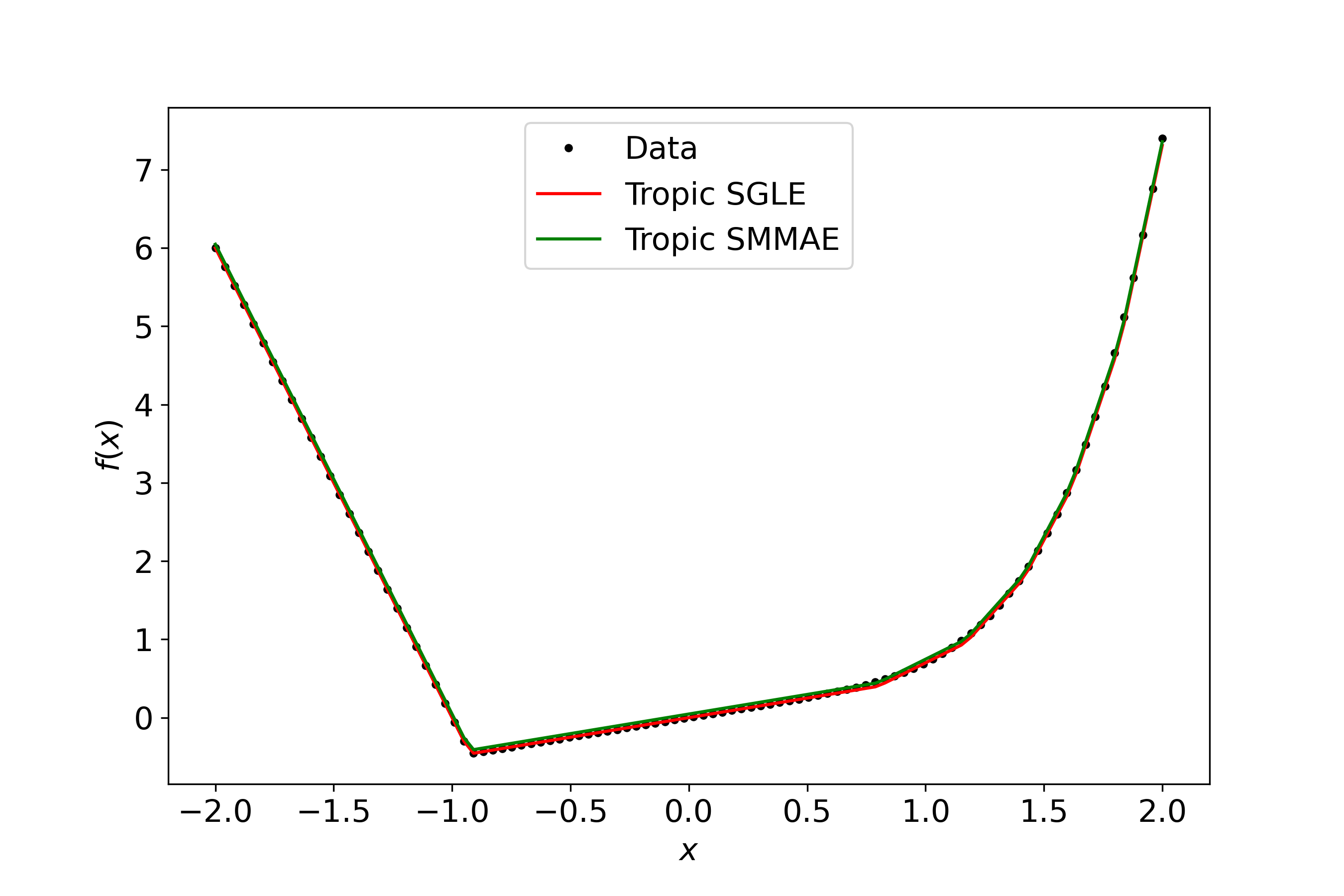}
            \caption{$K = 6, \epsilon = 0.0625, p = 2$}
        \end{subfigure}
        \begin{subfigure}{0.45\textwidth}
            \includegraphics[width=\textwidth]{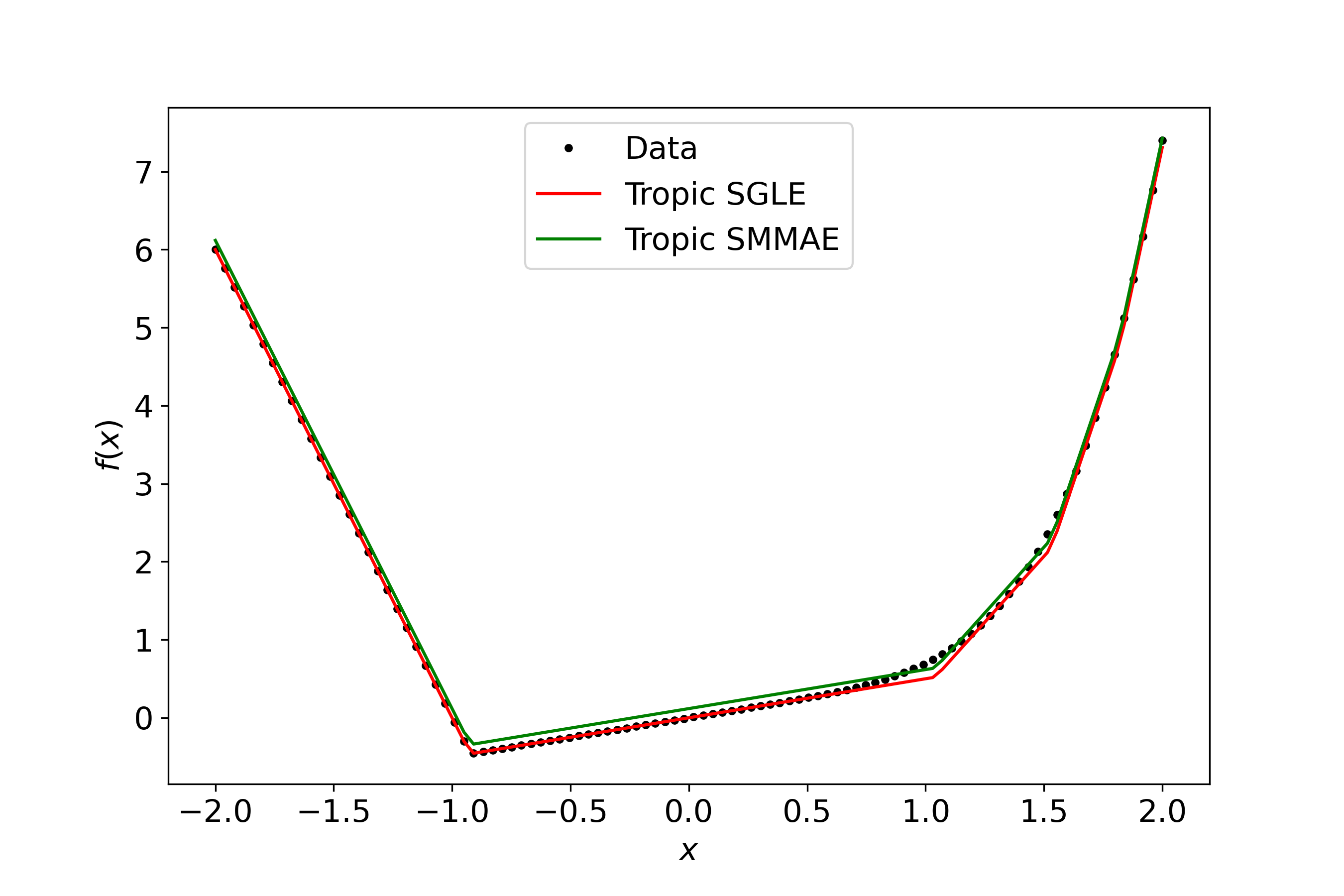}
            \caption{$K = 5, \epsilon = 1, p = 2$}
        \end{subfigure}
        \begin{subfigure}{0.45\textwidth}
            \includegraphics[width=\textwidth]{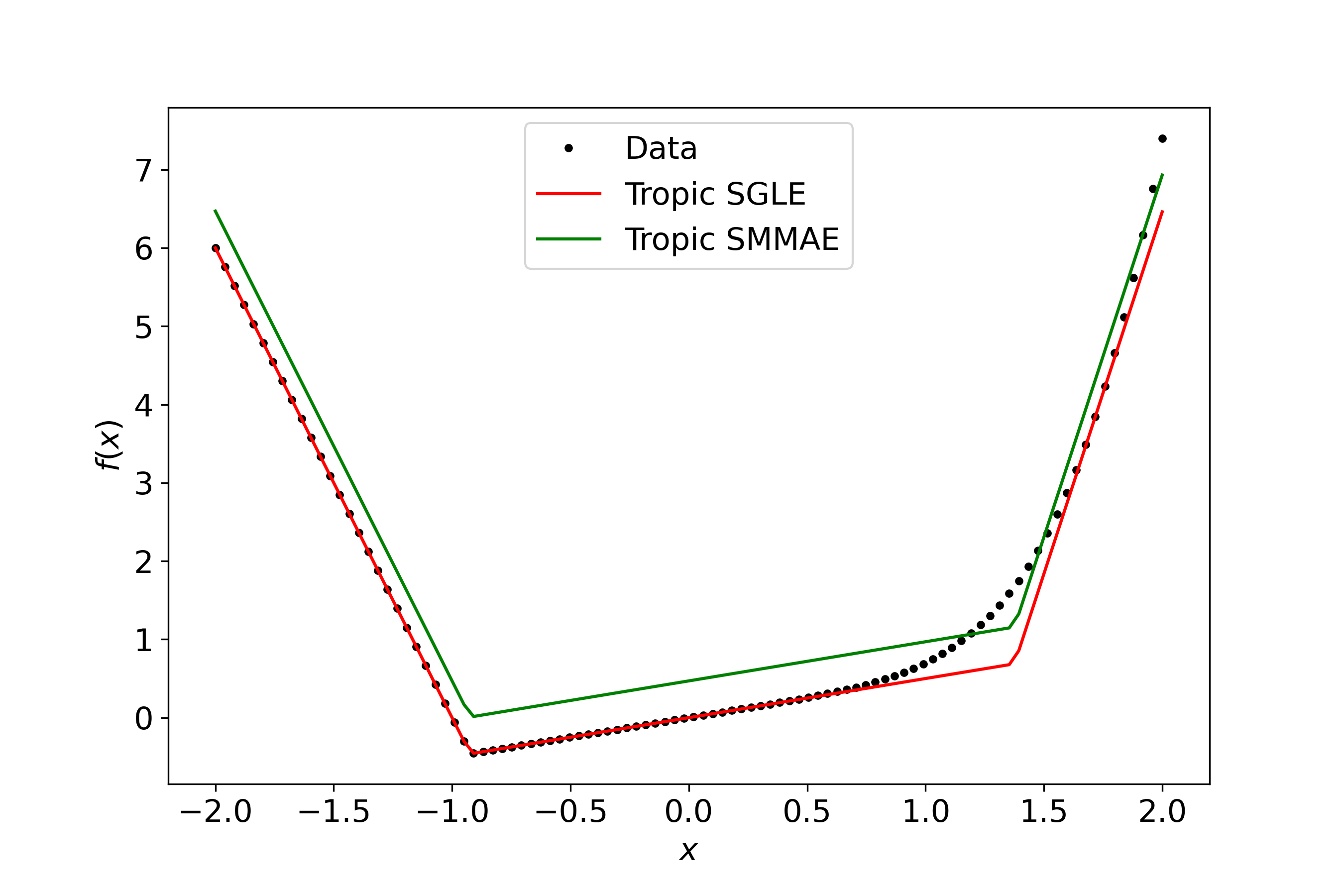}
            \caption{$K = 3, \epsilon = 1024, p = 5$}
        \end{subfigure}
        \caption{Piecewise linear approximations of $f(x) = \max(-6x-6, \frac{x}{2}, \frac{x^5}{5} + \frac{x}{2})$ with $K$ regions, resulting from the sparse tropical regression method with varied error threshold $\epsilon$ and norm order $p$. Best viewed in color.}
        \label{fig:1}
    \end{figure}
    \end{example}
    
    \begin{example}
    \begin{figure}
        \centering
        \begin{subfigure}{\textwidth}
            \includegraphics[width=0.85\textwidth, height=0.4\textheight]{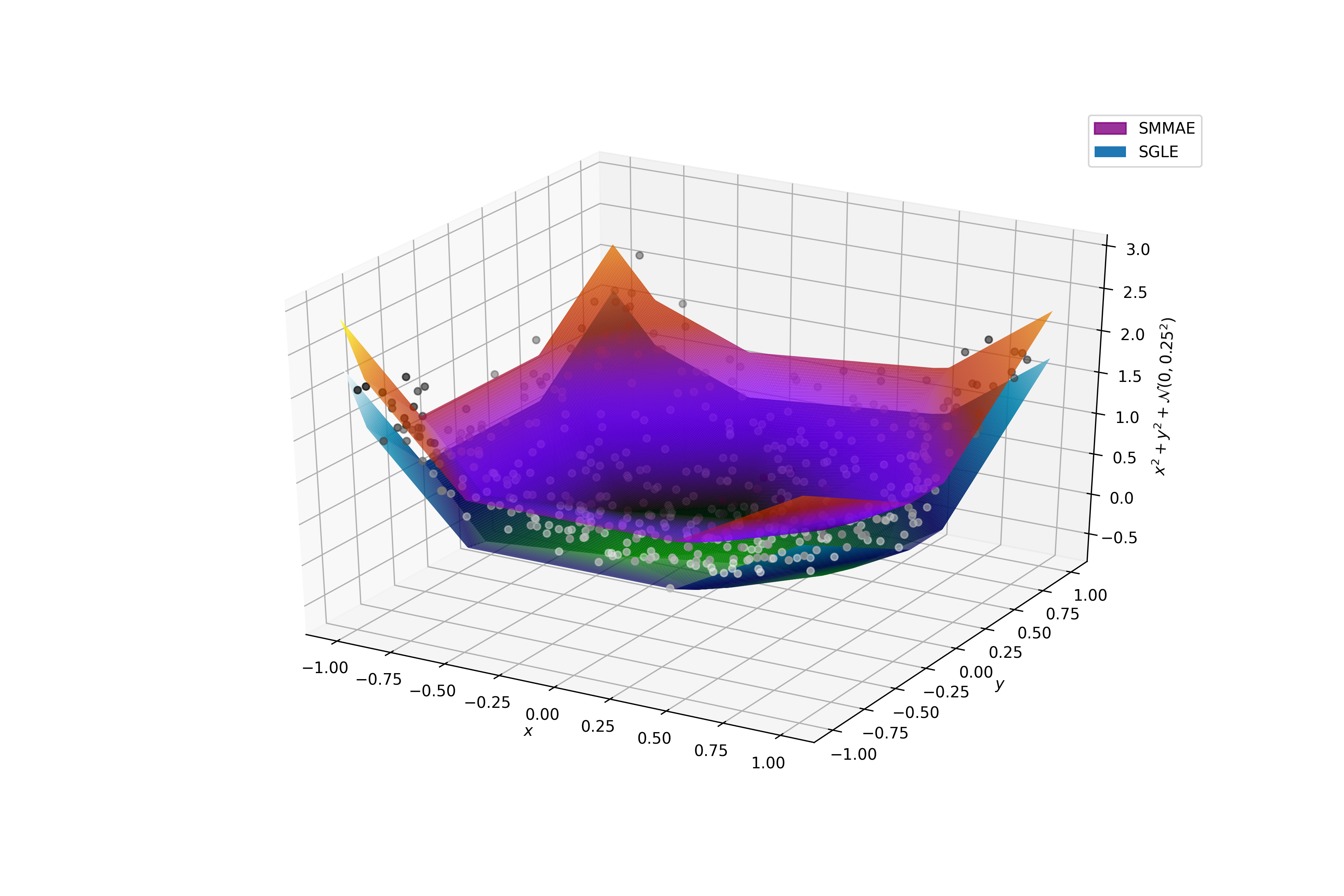}
            \caption{$K = 16, \epsilon = 10^8, p = 150$}
        \end{subfigure}
        \begin{subfigure}{\textwidth}
            \includegraphics[width=0.85\textwidth, height=0.4\textheight]{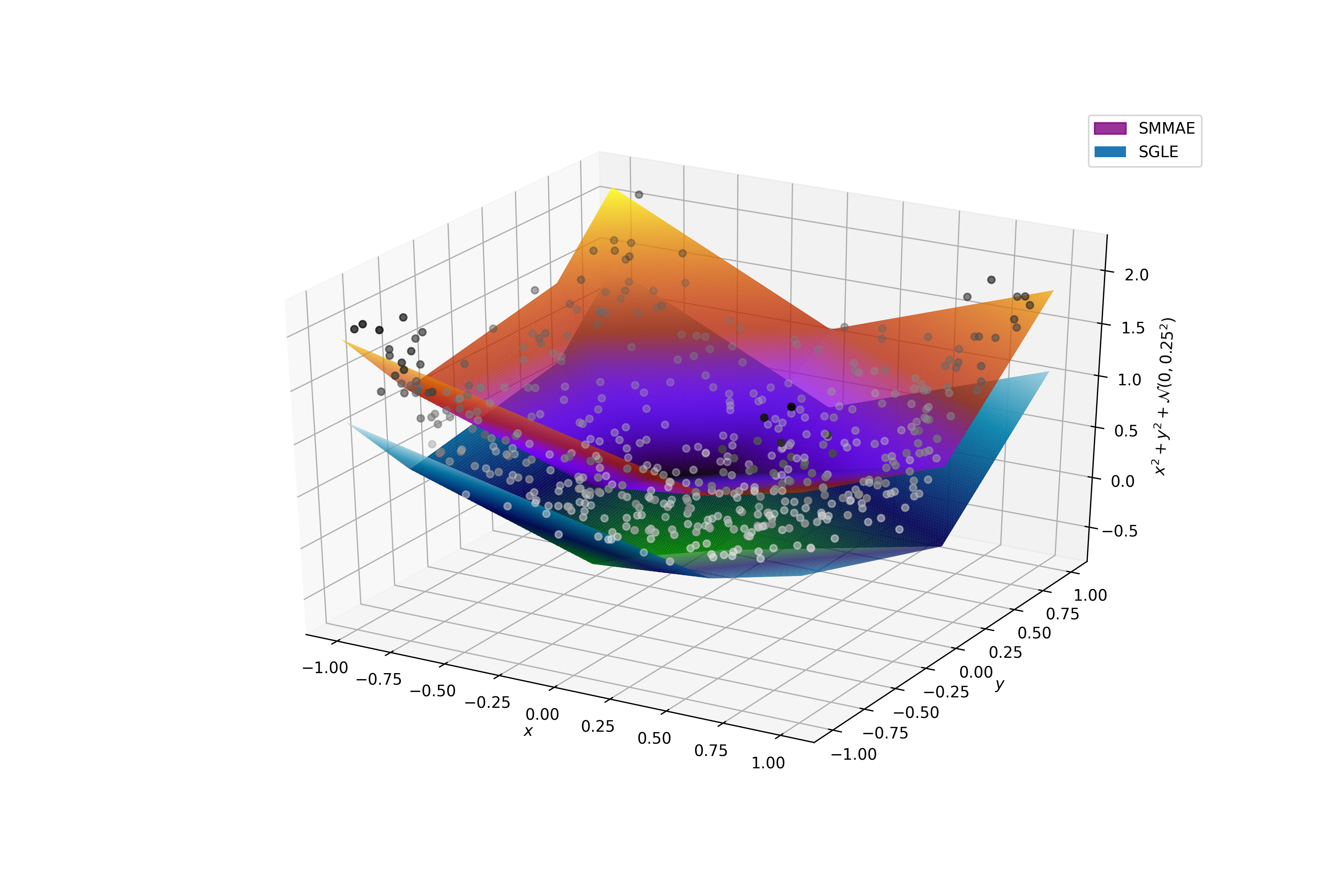}
            \caption{$K = 5, \epsilon = 220, p = 2$}
        \end{subfigure}
        \caption{The sparse greatest lower and minimum max absolute error estimates of surface $ z = x^2 + y^2 + \mathcal{N}(0, 0.25^2) $ for $2$ different runs of the fitting algorithm. Best viewed in color.}
        \label{fig:2}
    \end{figure}
    
    Let us now focus on the $2$-dimensional case, meaning we obtain data from a convex surface. For this example, we sample values from the noisy paraboloid surface:
    \begin{equation}\label{eq:18}
        z = x^2 + y^2 + \mathcal{N}(0, 0.25^2),
    \end{equation}
    where $x_i, y_i$ are drawn as i.i.d. random variables from the $ \text{Unif}[-1, 1] $ distribution. We obtain $500$ observations from the surface. \par
    Let $ A = \{-10.00, -9.75, -9.50, \ldots, 9.50, 9.75, 10\} $ be the set of the partial derivatives of the affine regions that are to be considered, then our tropical model for this example is
    \begin{equation}
        p(x, y) = \bigvee_{(k, l) \in A \times A}b_{kl} + kx + ly
    \end{equation}
    We obtain SGLEs and SMMAE estimators for the above model, with different runs of our algorithm, as in Example \ref{ex:1}. We present the results in Table \ref{tab:5}, compared to those obtained from the tropical regression method of \cite{MaTh20}, in which the number of affine regions is a pre-defined constant. Fig. \ref{fig:troVSstro} shows the RMS error of the SMMAE estimators as a function of the number $K$ of affine regions and compares it with the MMAE estimators reported in \cite{MaTh20}. \par
    We verify that, in the presence of noise, the SMMAE estimators perform better than the SGLEs, as the latter must approximate the data from below (See Fig. \ref{fig:2}) and, therefore, underestimate noise-corrupted low values. Both the estimators are able to find good approximations with a relatively low number of affine regions and the results are superior to those reported in \cite{MaTh20} (in terms of error and number of affine regions).  Notice that the SMMAE estimates have exactly half the $\ell_\infty$ error of their SGLEs counterparts, as expected by Proposition \ref{prop:SMMAE}. Moreover, observe that when $ p = 150 $, the SMMAE estimate has $\ell_\infty$ error equal to $0.5634$, which is very close to the theoretical upper bound from equation (\ref{eq:12}) ($\frac{10^{8/150}}{2} = 0.5653)$. This observation allows one to run targeted versions of the fitting algorithm (namely, choose a high order norm $p$ and set $ \epsilon = (2\delta)^p $, where $\delta$ is the accepted $\ell_\infty$ error threshold).
    
    \begin{figure}
        \centering
        \includegraphics[width=0.6\textwidth, height=0.4\textheight]{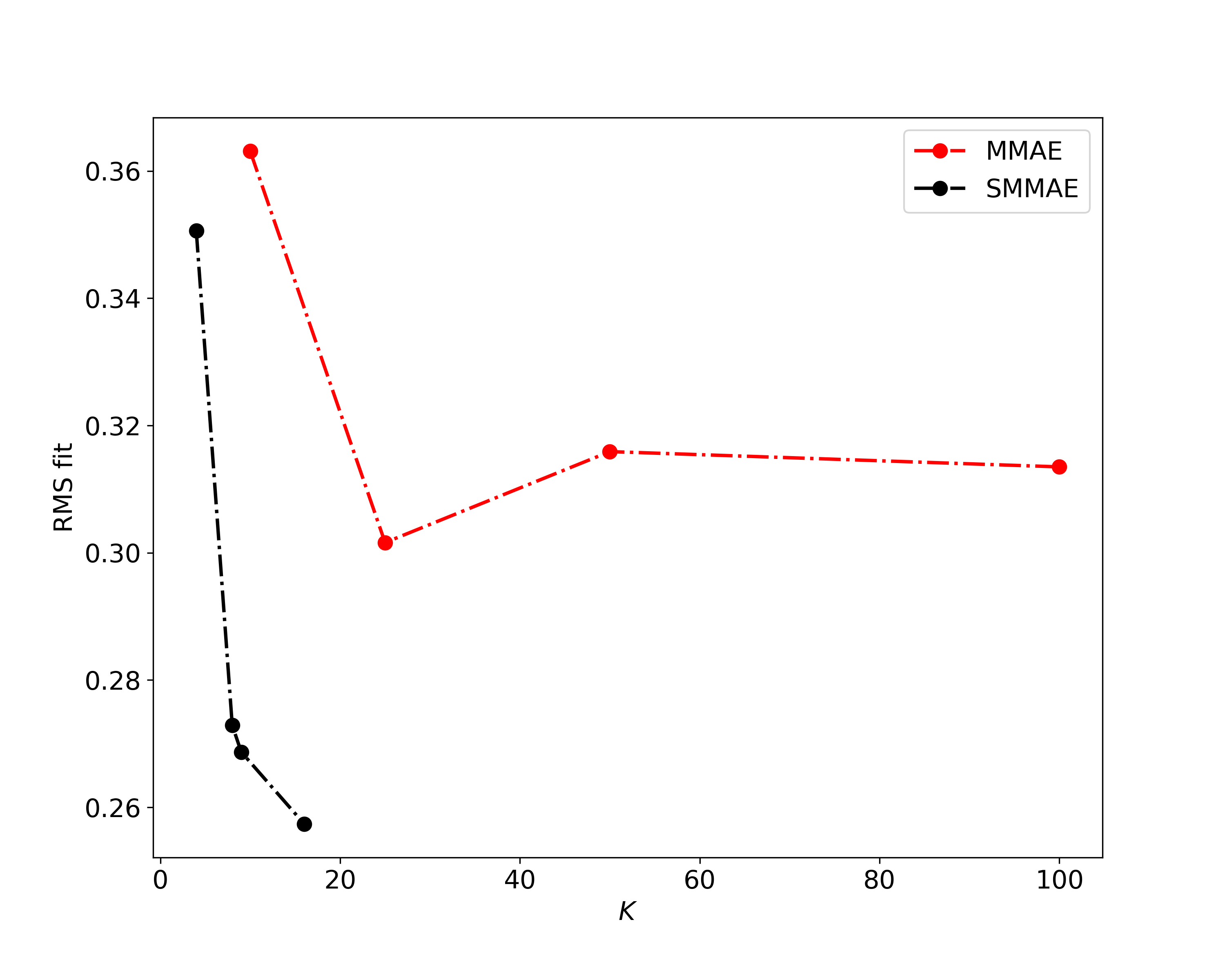}
        \caption{RMS error of SMMAE estimators vs number of affine regions $K$. Comparison between our method and the tropical regression method (MMAE) reported in \cite{MaTh20}.}
        \label{fig:troVSstro}
    \end{figure}
    
    \begin{table}
        \centering
        \begin{tabular}{|l||cc|cc|c|}
            \cline{2-5}
            \multicolumn{1}{}{}
            &
            \multicolumn{2}{|c|}{SGLE}
            &
            \multicolumn{2}{c|}{SMMAE}
            &
            \multicolumn{1}{}{}\\
            \hline
            $(\epsilon, p)$ & erro$\text{r}_{RMS}$ & erro$\text{r}_\infty$ & erro$\text{r}_{RMS}$ & erro$\text{r}_\infty$ & $|$supp$|$\\
            \hline
            $(210, 1)$ & $0.4926$ & $1.1575$ & $0.3027$ & $0.5787$ & $28$\\
            $(250, 1)$ & $0.5518$ & $1.1967$ & $0.2847$ & $0.5983$ & $8$\\
            $(300, 1)$ & $0.6681$ & $1.5405$ & $0.3506$ & $0.7703$ & $4$\\
            $(120, 2)$ & $\mathbf{0.4899}$ & $\mathbf{1.1268}$ & $0.2942$ & $\mathbf{0.5634}$ & $31$\\
            $(130, 2)$ & $0.5096$ & $1.1575$ & $0.2889$ & $0.5787$ & $16$\\
            $(150, 2)$ & $0.5465$ & $1.1734$ & $0.2729$ & $0.5867$ & $8$\\
            $(220, 2)$ & $0.6344$ & $1.5405$ & $0.3479$ & $0.7703$ & $5$\\
            $(360, 0.3)$ & $0.5050$ & $1.1390$ & $0.2956$ & $0.5695$ & $20$\\
            $(50, 5)$ & $0.5018$ & $\mathbf{1.1268}$ & $0.2812$ & $\mathbf{0.5634}$ & $23$\\
            $(75, 7)$ & $0.5602$ & $1.1963$ & $0.2687$ & $0.5981$ & $9$\\
            $(10^8, 150)$ & $0.5560$ & $\mathbf{1.1268}$ & $\mathbf{0.2574}$ & $\mathbf{0.5634}$ & $16$\\
            \hline
        \end{tabular}
        \begin{tabular}{|l||cc|cc|}
            \cline{2-5}
            \multicolumn{1}{}{}
            & 
            \multicolumn{2}{|c|}{GLE \cite{MaTh20}}
            &
            \multicolumn{2}{c|}{MMAE \cite{MaTh20}}\\
            \hline
            K & erro$\text{r}_{RMS}$ & erro$\text{r}_\infty$ & erro$\text{r}_{RMS}$ & erro$\text{r}_\infty$\\
            \hline
            $10$ & $0.6659$ & $1.6022$ & $0.3641$ & $0.8011$\\
            $25$ & $0.5674$ & $1.2779$ & $0.3016$ & $0.6389$\\
            $50$ & $0.5489$ & $1.3068$ & $0.3159$ & $0.6534$\\
            $100$ & $0.5364$ & $1.2828$ & $0.3135$ & $0.6414$\\
            \hline
        \end{tabular}
        \caption{PWL approximations and their errors of surface (\ref{eq:18}). $K$ is the number of affine regions in the resulting tropical polynomial.}
        \label{tab:5}
    \end{table}
        
    \end{example}
    
    \begin{example}\label{ex:3}
        Consider the case where dimension is $ n = 3 $ and we have $ m = 11^3 = 1331 $ points collected from the set $ V \times V \times V $, where $ V = \{-5, -4, \ldots, 4, 5\} $. The convex function to approximate is:
        \begin{equation}
            g(\mathbf{x}) = \log(\exp(x_1) + \exp(x_2) + \exp(x_3)).
        \end{equation}
        The above synthetic dataset was used before in the PWL fitting literature in \cite{MaBo09}. The authors propose an iterated method, which alternates between partitioning the data into affine regions and carrying out least squares fits to update the local coefficients. As the resulting approximation depends on the initial partition, the authors propose running multiple instances of their algorithm to obtain a good PWL fit to $g$. \par
        Note that when the dimension of the problem grows more than $n = 3$ or $n = 4$, it becomes infeasible to divide large $n$-dimensional intervals, $[-l, l]^n$, with a float step size, as $K$ becomes equal to $(\frac{2l+1}{step})^n$. Similar to \cite{MaTh20}, we propose instead finding the numerical gradients of the data, setting them as the candidate slopes $\mathbf{a}_k$ and then applying our tropical sparse method, to select some of the regions and determine their constant terms. By changing that, the method becomes tractable and grows as $\mathcal{O}(m^2)$. For this example, we fix $ p = 2 $ and to obtain the first approximation, we set $ \epsilon = 1331 $, so that the RMS error is less than $1$. The resulting tropical polynomial has $ K = 4 $ affine regions. From then on, we gradually lower $\epsilon$, so that we get approximations with varied $K$, until $K$ reaches $21$. Fig. \ref{fig:rmsVsK} shows the RMS errors versus the number of affine regions. The results are competitive to those reported in \cite{MaBo09}, while our method produces approximations with a single run, as opposed to \cite{MaBo09} which relies on 10 or 100 different trials, with complexity for each one of $\mathcal{O}((n+1)^2mi)$, $i$ being the number of iterations until convergence.
        
        \begin{figure}
            \centering
            \includegraphics[width=0.6\textwidth, height=0.4\textheight]{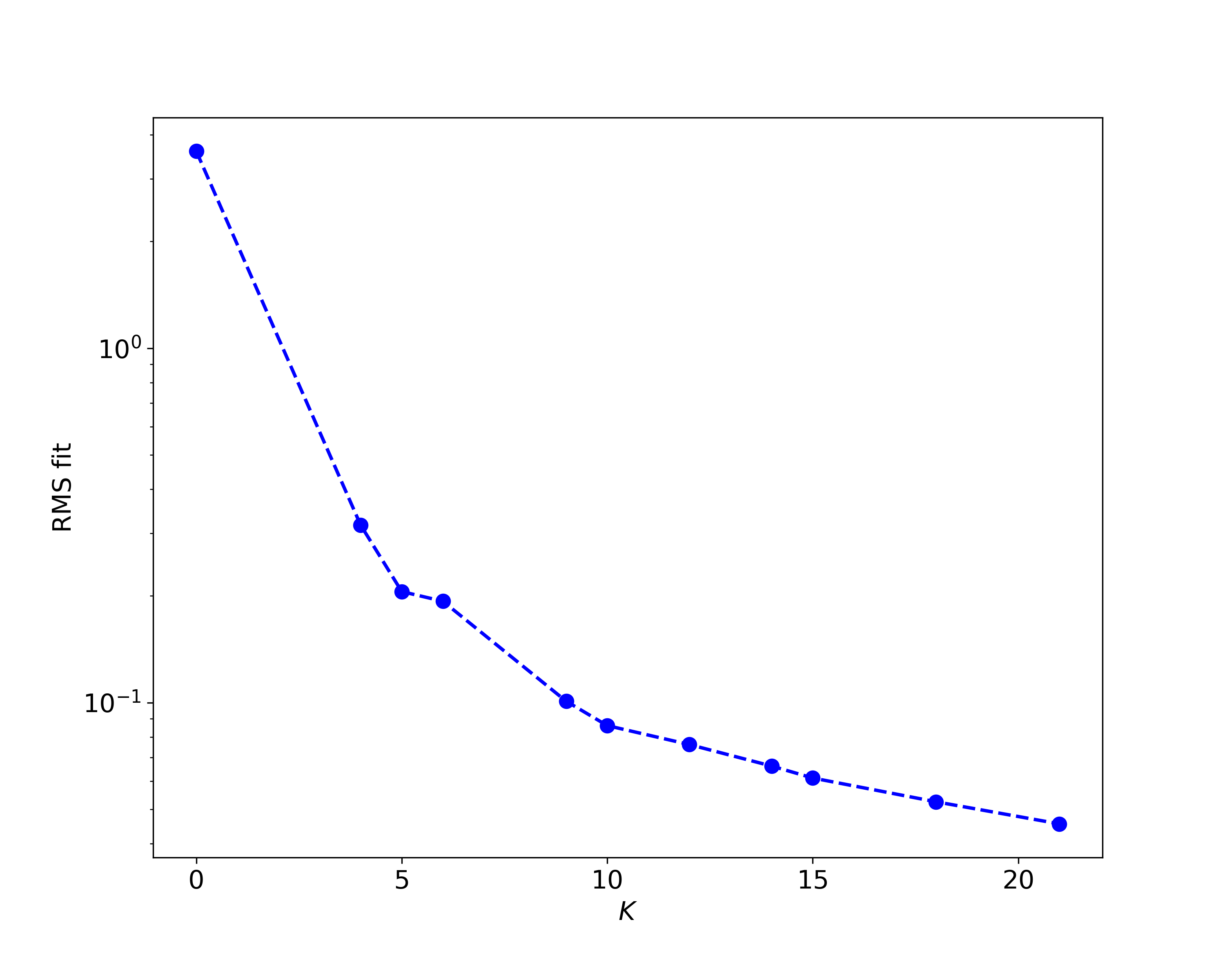}
            \caption{RMS error vs number of affine regions of PWL approximation of $ g(\mathbf{x}) = \log(\exp(x_1) + \exp(x_2) + \exp(x_3)) $.}
            \label{fig:rmsVsK}
        \end{figure}
    \end{example}

\section{Conclusions and Future Work}
Max-plus and tropical algebra serve as a framework for various fields, with emerging applications in optimization and machine learning. In this work, we demonstrated how to obtain sparse approximate solutions to max-plus equations and based on that, introduced a novel method for multivariate convex regression by PWL functions (i.e tropical regression) with a nearly optimal number of affine regions. The proposed method comes with error bounds for the resulting approximation and has an edge over previously reported tropical regression methods, in terms of robustness. In future work, we wish to further study the statistical properties of the tropical estimators, when dealing with noisy data. Lastly, an extension of the sparsity results in nonlinear vector spaces, called Complete Weighted Lattices \cite{Mara17}, would allow one to solve more general problems of regression, using the tools introduced in this work.

\section*{Acknowledgements} The authors wish to thank Manos Theodosis for helpful comments on this work.

\printbibliography

\end{document}